\documentclass{amsart}

\usepackage[english]{babel}
\usepackage[T1]{fontenc} 
\usepackage[latin1]{inputenc}

\usepackage{csquotes}
\usepackage[backend=bibtex, style=alphabetic, firstinits=true]{biblatex}
\usepackage[a4paper,top=3cm,bottom=3cm,left=3.5cm,right=3.5cm,
bindingoffset=5mm]{geometry}

\usepackage{amssymb}
\usepackage{amsthm}
\usepackage{mathrsfs}
\usepackage{stmaryrd}
\usepackage{paralist}

\usepackage{color}
\usepackage{hyperref}
\definecolor{darkblue}{RGB}{0,0,160}
\hypersetup{
	colorlinks,breaklinks,
	citecolor=black,
	filecolor=black,
	linkcolor=darkblue,
	urlcolor=darkblue
}

\usepackage[capitalize]{cleveref}

\newtheorem{Theorem}{Theorem}[section]
\newtheorem*{Theorem*}{Theorem}
\newtheorem{Lemma}[Theorem]{Lemma}
\newtheorem{Cor}[Theorem]{Corollary}

\theoremstyle{definition}
\newtheorem{Def}[Theorem]{Definition}
\newtheorem*{Def*}{Definition}

\newtheorem{Ex}[Theorem]{Example}

\newcommand{\kk}{k}            
\newcommand{\mmod}[1]{\mathrm{mod}(#1)}

\newcommand{\cmod}[1]{\mathrm{MCM}(#1)}

\newcommand{\Gl}[1]{\mathrm{GL}(#1)}   
\newcommand{\Sl}[1]{\mathrm{SL}(#1)}

\newcommand{\mm}{\mathfrak{m}}
\newcommand{\NN}{\mathbb{N}}
\newcommand{\ZZ}{\mathbb{Z}}

\newcommand{\RR}{\mathbb{R}}

\newcommand{\Ac}{\mathcal{A}}

\newcommand{\set}[1]{\{#1\}}
\newcommand{\with}{\,\colon\,}

\newcommand{\card}[1]{|#1|}

\DeclareMathOperator{\Syz}{Syz}
\DeclareMathOperator{\Sym}{Sym}

\DeclareMathOperator{\rank}{rank}
\DeclareMathOperator{\Hom}{Hom}
\DeclareMathOperator{\freerank}{frk}

\DeclareMathOperator{\rk}{rank}

\DeclareMathOperator{\Vol}{Vol}

\DeclareMathOperator{\mult}{mult}
\DeclareMathOperator{\chara}{char}

\title{The Symmetric signature of cyclic quotient singularities}

\author{Alessio Caminata and Lukas Katth\"an}

\address{{\small Alessio Caminata, Institut de Math\'{e}matiques, Universit\'{e} de Neuch\^{a}tel\\Rue Emile-Argand 11, CH-2000 Neuch\^{a}tel, Switzerland}}
\email{{\small alessio.caminata@unine.ch}}

\address{{\small Lukas Katth\"an, Institut f\"ur Mathematik, Goethe-Universit\"at Frankfurt, Robert-Mayer-Str. 10, 60325 Frankfurt am Main, Germany}}
\email{{\small katthaen@math.uni-frankfurt.de}}

\subjclass[2010]{Primary: 13A35; Secondary: 13A50, 52B20.}

\keywords{F-signature; quotient singularity; toric ring; Auslander correspondence.}

\begin{document}

\begin{abstract}
The symmetric signature is an invariant of local domains which was recently introduced by Brenner and the first author in an attempt to find a replacement for the $F$-signature in characteristic zero.
In the present note we compute the symmetric signature for two-dimensional cyclic quotient singularities, i.e. invariant subrings $\kk\llbracket u,v\rrbracket^G$ of rings of formal power series under the action of a cyclic group.
Equivalently, these rings arise as the completions (at the irrelevant ideal) of two-dimensional normal toric rings.
We show that for this class of rings the symmetric signature coincides with the $F$-signature.
\end{abstract}

\maketitle

\section*{Introduction}
The \emph{$F$-signature} \cite{HL02} or \emph{minimal relative Hilbert-Kunz multiplicity} \cite{WY04} is an important invariant of local rings of positive characteristic.
For a Noetherian local domain $(R,\mathfrak{m},\kk)$ of positive characteristic with algebraically closed residue field $\kk$, the $F$-signature of $R$ is defined as the limit
\[
s(R):=\lim_{e\rightarrow+\infty}\frac{\freerank_R{^e\!R}}{\rank_R{^e\!R}},
\]
where $^e\!R$ denotes the ring $R$ viewed as $R$-module with multiplicative structure twisted via the $e$-th iterated Frobenius homomorphism.
The notation $\rank_R(-)$ stands for the usual rank of an $R$-module, and we have that $\rank_R{^e\!R}=p^{de}$, where $p$ is the characteristic of the ring, and $d$ the Krull dimension.
The number $\freerank_R{^e\!R}$ is the free rank of $^e\!R$, that is the maximum number of free direct $R$-summands of $^e\!R$ (see Definition \ref{definitionfreerank}).
The $F$-signature exists for every $F$-finite ring \cite{Tuc12}, and it is equal to $1$ if and only if the ring is regular \cite{WY00}. In general, the $F$-signature provides delicate information concerning the singularities of the ring.

For some classes of rings, the $F$-signature does not depend on the characteristic (in a certain sense).
This motivated Brenner and the first author to introduce the \emph{symmetric signature} as an analogous invariant also for rings of characteristic zero \cite{BC16}. 
They consider $q$-th reflexive symmetric powers $\mathcal{S}^q$ of the top-dimensional syzygy module of the residue field and they study the asymptotic behaviour of the ratio $\freerank_R/\rank_R$ for this class of modules, rather than $^e\!R$ as for the $F$-signature.
If the limit
\[ 
\lim_{N\rightarrow+\infty} \frac{\sum_{q=0}^N\freerank_R\mathcal{S}^q}{\sum_{q=0}^N\rank_R\mathcal{S}^q}
\]
exists, it is called \emph{symmetric signature} of $R$ and denoted by $s_{\sigma}(R)$ (see Definition \ref{defsymmetricsignature}).

We consider the symmetric signature for a special class of rings. A \emph{quotient singularity} is a ring of the form $R = \kk\llbracket x_1,\dots,x_n\rrbracket^G$ where $G$ is a small finite subgroup of $\Gl{n,\kk}$ whose order is not a multiple of the characteristic of $\kk$.
It is known that the $F$-signature of a quotient singularity is $\frac{1}{\card{G}}$, cf. \cite{WY04}.
Brenner and the first author showed that the symmetric signature takes the same value in the case $n = 2$ and $G \subset \Sl{n,\kk}$, i.e. for two-dimensional \emph{special} quotient singularities.

The main goal of this paper is to extend this result to two-dimensional \emph{cyclic} quotient singularities, i.e. the case $n = 2$ and $G \subset \Gl{2,\kk}$ is cyclic.
In fact we prove a slightly stronger statement, namely we compute the \emph{generalized symmetric signature} of these rings (see \Cref{def:gensymsig} below).

The structure of the paper is the following. 
In Section \ref{sectionsymmetricsignature} we review the definition of symmetric signature and its generalized version.
In Section \ref{section:quotientsingularities} we state our main theorem about the generalized symmetric signature of cyclic quotient singularities (Theorem \ref{thm:gensymsig}), and we use the Auslander correspondence to reduce it to a problem concerning group representations.
This is proved in Section \ref{section:representations} by counting lattice points in a certain polyhedron.

\section{The Symmetric signature}\label{sectionsymmetricsignature}

\par Let $(R,\mathfrak{m},\kk)$ be a Noetherian local domain of Krull dimension $d$, we recall the definition of symmetric signature from \cite{BC16}.
Given a finitely generated $R$-module $M$, we denote by $\rk_R M$ the rank of $M$, and we introduce a notation for the number of free summands.

\begin{Def}\label{definitionfreerank}
	The \emph{free rank} of $M$ is
	\[
		\freerank_R(M):=\max\{n: \ \exists \text{ a split surjection }\varphi: M\twoheadrightarrow R^n\}. 
	\]
\end{Def}
If the module $M$ has finite rank, we can write it as $M\cong R^{\freerank_R(M)}\oplus N$, where the module $N$ has no free direct $R$-summands. It follows that in general $\freerank_R{M}\leq\rank_RM$.

\par For a natural number $q$, we denote by $\Sym^q_R(M)$ the $q$-th symmetric powers of the module $M$, and by $M^{*}:=\Hom_R(M,R)$ the dual of $M$.
The $R$-module $M^{**}$ is always reflexive and is called reflexive hull of $M$. The functor which sends $M$ to $\Sym^q_R(M)^{**}$ is called $q$-th reflexive symmetric powers.

\par We consider the top-dimensional syzygy module of the residue field $\Syz^d_R(\kk)$. If the ring $R$ is Cohen-Macaulay, then $\Syz^d_R(\kk)$ is a maximal Cohen-Macaulay (MCM for short) module. Let $q$ be a natural number, we apply $q$-th reflexive symmetric powers to $\Syz_R^d(\kk)$, and we obtain a family of reflexive $R$-modules 
\[
	\mathcal{S}^q:=\left(\Sym_R^q\big(\Syz^d_R(\kk)\big)\right)^{**}.
\]
\begin{Def}\label{defsymmetricsignature}
	The number 
	\[ 
		s_{\sigma}(R) := \lim_{N\rightarrow+\infty} \frac{\sum_{q=0}^N\freerank_R\mathcal{S}^q}{\sum_{q=0}^N\rank_R\mathcal{S}^q}
	\] 
	is called \emph{symmetric signature} of $R$, provided the limit exists.
\end{Def}

\begin{Ex}
	If $R$ is a regular ring, then the modules $\mathcal{S}^q$ are $R$-free. 
	It follows that $\freerank_R\mathcal{S}^q=\rank_R\mathcal{S}^q$ for all $q$, and consequently $s_{\sigma}(R)=1$.
\end{Ex}

Assume now that the category of finitely generated reflexive $R$-modules is a Krull-Schmidt category.
In this case we can also define a generalized symmetric signature.
For reflexive $R$-modules $M, M'$ where $M$ is indecomposable, we denote by
\[ \mult(M,M') \]
the \emph{multiplicity of $M$ in $M'$}, i.e. the number of summands isomorphic to $M$ in a decomposition of $M'$ into indecomposable modules.
\begin{Def}\label{def:gensymsig}
	The number 
	\[
		s_{\sigma}(R,M) := \lim_{N\rightarrow+\infty} 
		\frac{\sum_{q=0}^N \mult(M, \mathcal{S}^q)}{\sum_{q=0}^N \rank_R\mathcal{S}^q}
	\]
	is called \emph{generalized symmetric signature of $R$ with respect to $M$}, provided the limit exists.
\end{Def}

For further examples and properties of the symmetric signature, we refer to the paper \cite{BC16}.
We mention that there is also an alternative version of the symmetric signature, called \emph{differential symmetric signature}, where the module $\Syz^d_R(\kk)$ is replaced by the module $\Omega_{R/\kk}^{**}$ of Zariski differentials of $R$ over $\kk$.
The two symmetric signatures share many properties, but we will not consider the differential variant in this paper.

\section{Quotient singularities}\label{section:quotientsingularities}
Let $\kk$ be an algebraically closed field $\kk$, and let $G\subseteq\Gl{2,\kk}$ be a finite group, whose order is not a multiple of the characteristic of $\kk$.
Assume further that $G$ is small, i.e. it contains no pseudo-reflections.

The group $G$ acts on $S=\kk\llbracket u,v\rrbracket$ via linear changes of variables, and we denote by $R:=S^G$ the invariant ring under this action.
The invariant ring $R$ is called \emph{quotient singularity} and has the following properties:
it is a Noetherian local domain of dimension $2$, normal, Cohen-Macaulay, complete, and an isolated singularity (see e.g. \cite[Theorem 4.1]{BD08}).
In particular, $R$ has finite Cohen-Macaulay type, i.e. there exist only finitely many isomorphism classes of maximal Cohen-Macaulay $R$-modules.
We denote by $\cmod{R}$ the category of MCM $R$-modules, which is a Krull-Schmidt category since $R$ is complete.
The ring $R$ is Gorenstein if and only if $G\subseteq\Sl{2,\kk}$ (see \cite{Wat74a, Wat74b}), and in this case $R$ is called \emph{special quotient singularity}.
If $G$ is a cyclic group, then $R$ is called \emph{cyclic quotient singularity}.
Note that cyclic quotient singularities are exactly the completion of two-dimensional normal toric rings at their irrelevant ideal.
(If $R$ is a cyclic quotient singularity, then we may assume that $G$ acts diagonally, so $R$ is the completion of a subalgebra of $\kk[u,v]$ which is generated by monomials and hence toric.
On the other hand, if $R$ is a two-dimensional normal ring, then the so-called standard map \cite[p. 55]{BG} yields a presentation of $R$ as invariant subring of $\kk[u,v]$ under the action of some group $G$.)

Hashimoto and Nakajima \cite{HN15} proved that if $\kk$ has prime characteristic and $M$ is an indecomposable MCM $R$-module, then the generalized $F$-signature of $R$ with respect to $M$ is
\[
s(R,M)=\frac{\rank_R M}{\card{G}}.
\]

Brenner and the first author \cite{BC16} proved that the generalized symmetric signature has the same value in the case of special quotient singularities.
The main result of the present paper is the corresponding statement for cyclic quotient singularities:

\begin{Theorem}\label{thm:gensymsig}
	Let $\kk$ be an algebraically closed field and let $G \subset \Gl{2,\kk}$ be a small, finite and cyclic group, such that $\chara\kk$ does not divide $\card{G}$. 
	Let $R = \kk\llbracket u,v\rrbracket^G$ be the corresponding cyclic quotient singularity and let $M$ be an indecomposable MCM $R$-module. 
	Then the generalized symmetric signature of $R$ with respect to $M$ equals
	\[
		s_{\sigma}(R,M)=\frac{1}{\card{G}}.
	\]
\end{Theorem}
To prove this result, we use the Auslander correspondence to reduce it into a group theoretical statement (see \Cref{cor:auslander,lemmacyclicrepresentationfaithful}), which is then proven in \Cref{theoremcyclicrepresentation}.
If we choose $M=R$ in the previous theorem, we immediately get the following corollary.

\begin{Cor}\label{corollarycyclicsingularity}
	Let $\kk$ be an algebraically closed field and let $G \subset \Gl{2,\kk}$ be a small, finite and cyclic group, such that $\chara\kk$ does not divide $\card{G}$. 
	Then the symmetric signature of the corresponding cyclic quotient singularity $R = \kk\llbracket u,v\rrbracket^G$ is
	\[
	s_{\sigma}(R)=\frac{1}{\card{G}}.
	\]
\end{Cor}

\subsection{The Auslander correspondence}
There is a one-to-one correspondence, called Auslander correspondence, between irreducible $\kk$-representations of $G$ and indecomposable MCM $R$-modules.
More precisely, we a have a functor $\Ac:\mmod{\kk[G]}\rightarrow\cmod{R}$ called \emph{Auslander functor}, from the category $\mmod{\kk[G]}$ of linear $\kk$-representations of $G$ to the category of MCM $R$ modules, given by $\Ac(V):=(S\otimes_{\kk}V)^G$.
This functor is not an equivalence of categories, but it has a right adjoint $\Ac'$ given by $\Ac'(N):=(S\otimes_RN)^{**}\otimes_{S}\kk$, for every MCM $R$-module $N$.

We collect some properties of the functors $\Ac$ and $\Ac'$ in the following theorem.

\begin{Theorem}[Auslander correspondence]\label{Auslandertheorem}
	The functors $\Ac:\mmod{\kk[G]}\rightarrow\cmod{R}$ and  $\Ac':\cmod{R}\rightarrow\mmod{\kk[G]}$ have the following properties.
	\begin{compactenum}[1)]
		\item $\Ac(V)$ is indecomposable in $\cmod{R}$ if and only if $V$ is an irreducible representation.
		\item $\Ac(\Ac'(M)) \cong M$ for every $M \in \cmod{R}$.
		\item $\Ac'(\Ac(V)) \cong V$ for every $V \in\mmod{\kk[G]}$.
		\item $\rank_R\Ac(V)=\dim_{\kk}V$ for every $\kk$-representation $V$.
		\item If $V, V' \in \mmod{\kk[G]}$ and $V'$ is irreducible, then 
		\[\mult(V', V) = \mult(\Ac(V'), \Ac(V)).\]
	\end{compactenum}
\end{Theorem}

We recall also that for a two-dimensional quotient singularity finitely generated MCM $R$-modules coincide with finitely generated reflexive $R$-modules,
and the indecomposable MCM $R$-modules are exactly the $R$-direct summands of $S$ by a result of Herzog \cite{Her78}.

\par Brenner and the first author \cite{BC16} proved  that the Auslander functor commutes with reflexive symmetric powers.
Note that reflexive symmetric powers of a $\kk$-representation are just ordinary symmetric powers, since every vector space is canonically isomorphic to its double dual.

\begin{Theorem}[\cite{BC16}, Theorem 2.12]\label{symcommuteswithauslandertheorem}
	For every $\kk$-representation $V$ of $G$, it holds that
	\[\Sym_R^{q}(\Ac(V))^{**}\cong\Ac(\Sym^{q}_{\kk}(V)). \]
\end{Theorem}

For further details, and a review of the Auslander correspondence the reader may consult the books of Leuschke and Wiegand \cite{LW12}, and of Yoshino \cite{Yos90}, or the first author's Ph.D. thesis \cite{Cam16}.

In the present paper, we will use the following consequence of \Cref{symcommuteswithauslandertheorem}.
\begin{Cor}\label{cor:auslander}
	Let $M, M'$ be MCM $R$-modules and assume that $M'$ is indecomposable.
	Then the following holds:
	\begin{enumerate}
		\item $\rank_R \Sym_R^q(M)^{**} = \dim_\kk \Sym_\kk^q(\Ac'(M))$, and
		\item $\mult(M', \Sym_R^q(M)^{**}) = \mult( \Ac'(M'),  \Sym_\kk^q(\Ac'(M)))$.
	\end{enumerate}
\end{Cor}
\begin{proof}
	The first item follows from the fact that 
	\[ \Sym_R^q(M)^{**} \cong \Sym_R^q(\Ac(\Ac'(M)))^{**} \cong \Ac(\Sym_\kk^q(\Ac'(M))) \]
	and part 4) of \Cref{Auslandertheorem}.
	For the second equality, we compute that
	\begin{align*}
		\mult(M', \Sym_R^q(M)^{**}) &= \mult(\Ac(\Ac'(M')), \Ac(\Sym_\kk^q(\Ac'(M)))) \\
		&= \mult(\Ac'(M'), \Sym_\kk^q(\Ac'(M))),
	\end{align*}
	where we used the last part of \Cref{Auslandertheorem}.
\end{proof}

By this corollary, to prove \Cref{thm:gensymsig} we need to consider the multiplicities of irreducible representation of $G$ inside the symmetric powers of a given representation of $G$.
This will be done in the next section.
Before we turn to this, we need to show that the relevant representation is faithful:
\begin{Lemma}\label{lemmacyclicrepresentationfaithful}
	Let $G$ and $R$ as in \Cref{thm:gensymsig} and assume that $G$ is cyclic.
	Then the representation of $G$ on $\Ac'(\Syz^2_R(\kk))$ is faithful.
\end{Lemma}
\begin{proof}
	Let $n := \card{G}$ and let $g \in G$ be a generator for $G$.
	After a change of coordinates in $S = \kk\llbracket u,v\rrbracket$,
	we may assume that $g$ acts as $g \cdot u = \xi u$ and $g \cdot v = \xi^a v$ for some primitive $n$-th root of unity $\xi \in \kk$ and $a \in \NN, a < n$.
	Moreover, $G$ being small implies that $a$ is coprime to $n$.
	
	Every monomial in $S$ is an eigenvector under the action of $G$, hence the monomials invariant under $G$ form a $\kk$-basis of $R$. These are exactly the monomials $u^iv^j$ with $i + aj \equiv 0 \mod{n}$.
	In particular, the maximal ideal $\mm_R$ of $R$ has a minimal system of generators $p_1,\dots,p_{\mu}$ consisting of monomials.
	We may assume that $p_1 = u^n$ and $p_2 = u^{n-a} v$.
	
	We consider $p_1,\dots,p_{\mu}$ as elements of $S$ and let $\phi: S^\mu \to S$ be the map sending $e_i$ to $p_i$.
	Let $F$ and be the kernel of $\phi$ and let $N := S/(p_1,\dots,p_{\mu})$ be the cokernel of $\phi$.
	
	As $S$ is a regular ring of dimension $2$, the exact sequence
	\begin{equation}\label{S*Gimportantsequence}
	0 \rightarrow F \rightarrow S^\mu \xrightarrow{\phi} S\rightarrow N\rightarrow 0
	\end{equation}
	is a free resolution of $N$. In particular, $F$ is a free $S$-module.
	As $F$ is free, it is of the form $S \otimes_\kk V$ for some representation $V$ of $G$.
	
	We apply the exact $G$-invariant functor $(\_)^G$ to sequence \eqref{S*Gimportantsequence} and we get
	\[
	0\rightarrow M\rightarrow R^\mu \rightarrow R\rightarrow \kk\rightarrow0.
	\]
	Since $p_1,\dots,p_{\mu}$ is a system of generators of the maximal ideal of $R$, we obtain a copy of the residue field $\kk$ in the last position.
	It follows that the module $M$ is the second syzygy of the residue field, in other words $\Syz^2_R(\kk) = M = (S\otimes_{\kk}V)^G = \Ac(V)$ and hence $\Ac'(\Syz^2_R(\kk)) = V$.
	
	In order to show that $G$ acts faithfully on $V$, it is enough to show that there exists an element $s \in V$ with a trivial stabilizer.
	For this we consider the element $s := v e_1 - u^{a} e_2 \in S^\mu$.
	Clearly $\phi(s) = v p_1 - u^{a} p_2 = 0$, so $s \in F$.
	Since $(p_1, \dotsc, p_{\mu}) \subset S$ is a monomial ideal in two variables, it follows from Proposition 3.1 of \cite{MS05} that $s$ is a minimal generator of $F$, so it corresponds to a non-zero element of $V$.
	The action of $g$ on $s$ is given by
	\[
	g \cdot s = \xi^{a} v p_1 - \xi^a u^{a} p_2 = \xi^a s
	\]
	Therefore, as $a$ is coprime with $n$, the element $s$ has a trivial stabilizer and hence $G$ acts faithfully on $V$.
\end{proof}

\section{Representations and Lattice points}\label{section:representations}
In view of \Cref{cor:auslander} we need to consider multiplicities of irreducible representations of an abelian group inside symmetric powers of a given representation.
This is achieved in the following theorem.

\begin{Theorem}\label{theoremcyclicrepresentation}
	Let $\kk$ be an algebraically closed field and $G$ be a finite abelian group whose order is not a multiple of $\chara \kk$.
	Let $V, V'$ be a faithful and an irreducible $\kk$-representation of $G$, respectively.
	Then it holds that
	\[
		\lim_{N\rightarrow+\infty} \frac{\sum_{q=0}^N \mult(V', \Sym^q(V))}
		{\sum_{q=0}^N \dim_\kk \Sym^q(V)} = \frac{1}{\card{G}}.
	\]
\end{Theorem}

For the proof of the theorem, we first prove two lemmas. We start with an elementary result from group theory.

\begin{Lemma}\label{lemmacyclicrepresentation1}
	Let $n, \nu \in \NN$ and let $H\subseteq(\ZZ/n)^{\nu}$ be a subgroup. Let further $L$ be the kernel of the map
	\[
		\ell:\ZZ^\nu\rightarrow\Hom_{\ZZ}(H,\ZZ/n)
	\]
	given by $x\rightarrow\left(y\mapsto\sum_ix_iy_i\right)$. Then $[\ZZ^\nu:L]=\card{H}$. 
\end{Lemma}

\begin{proof}
	We first show that $\card{H} = \card{\Hom_{\ZZ}(H, \ZZ/n)}$. $H$ is a finitely generated abelian group, so it splits into a direct sum $H = \bigoplus_i H_i$ of cyclic groups. Moreover, every $H_i$ has order divisible by $n$.
	The generator of a $H_i$ of order $u$ can be mapped to any element whose order divides $u$ in $\ZZ/n$, and there are exactly $u$ such elements.
	Hence
	\[
	\card{H} = \prod_i \card{H_i} = \prod_i \card{\Hom_{\ZZ}(H_i, \ZZ/n)} = \card{\Hom_{\ZZ}(H, \ZZ/n)}.
	\]
	\par Next we show that $\ell$ is surjective.
	We denote  by $ \langle -, -\rangle$ the standard scalar product, i.e. $\ell(x)(y) = \sum_i x_i y_i = \langle x,y \rangle$.
	Let $e_1, \dotsc, e_\nu$ be the unit vectors of $(\ZZ/n)^\nu$.
	By the elementary divisor theorem, there exists an invertible matrix $A \in (\ZZ/n)^{\nu\times \nu}$, elements $\alpha_1, \dotsc, \alpha_r \in \ZZ/n$, $r \leq \nu$, and generators $h_1, \dotsc, h_r$ of $H$, such that $A h_i := \alpha_i e_i$ for $1 \leq i \leq r$.
	It is clear that $\Hom_{\ZZ}(H, \ZZ/n)$ is generated by the maps $\varphi_i$ sending $h_j$ to $\alpha_j$ for $j=i$ and all other generators to zero. 
	But the map $\varphi_i$ is the image under $\ell$ of (an arbitrary lifting to $\ZZ^\nu$ of) the vector $A^t e_i$, where $A^t$ denotes the transpose of $A$:
	\begin{align*}
		\ell(A^t e_i)(h_j) &= \langle A^t e_i,h_j \rangle = \langle e_i, A h_j \rangle \\
		&= \langle e_i,\alpha_j e_j \rangle = \alpha_j \delta_{ij}= \varphi_i(h_j),
	\end{align*}
	where $\delta_{ij}$ is the Kronecker symbol.
	So the claim follows.
\end{proof}

\begin{Lemma}\label{lemmacyclicrepresentation2}
	Let $G, V$ and $V'$ be as in \Cref{theoremcyclicrepresentation}, and let $\nu:=\dim_{\kk}V$.
	Then $\dim_\kk \Sym^q(V) = \card{\set{x \in \NN^\nu \with |x| = q}}$ and
	there exists a lattice $L\subseteq \ZZ^\nu$ and an $a_0\in\ZZ^\nu$ such that
	\[
		\mult(V', \Sym^q(V)) = \card{\set{x \in \NN^\nu \cap (a_0 + L) : |x| = q}}.
	\]
	and $[\ZZ^\nu:L]=\card{G}$.
	Here $|x| := \sum_i x_i$.
\end{Lemma}

\begin{proof}
	The representation $V$ splits into one-dimensional representations $V = \bigoplus_i V_i$.
	Let $x_i$ be a basis vector for $V_i$, then $\Sym V$ can be identified with the polynomial ring $\kk[x_1, \dotsc, x_\nu]$.
	It is clear that $\dim_k \Sym^q(V)$ equals the number of monomials of degree $q$, so the claimed formula holds.
	
	$G$ acts faithfully and diagonally on $V = \bigoplus_i V_i$, so we can identify it with a subgroup $G_1 \subseteq (\kk^*)^\nu$.
	The action of $G_1$ is now given by 
	\[
		g \cdot \underline{x}^{\underline{a}} = g^{\underline{a}} \underline{x}^{\underline{a}}, 
	\]
	where $\underline{x}^{\underline{a}} = \prod_i x_i^{a_i}$ is a monomial and $g^{\underline{a}} := \prod_i g_i^{a_i}, g = (g_1, \dotsc, g_\nu)$.
	
	For each point $\underline{a}\in\ZZ^\nu$, we have a one-dimensional representation $\rho_{\underline{a}}:G_1\rightarrow\kk^*$, $\rho_{\underline{a}}(g)=g^{\underline{a}}$. 
	So, we obtain a map
	\begin{align*}
		\ell: \ZZ^\nu &\rightarrow \Hom_{\ZZ}(G_1, \kk^*)\\
		\underline{a} &\mapsto \rho_{\underline{a}}.
	\end{align*}
	Two points $\underline{a}$ and $\underline{b}$ are mapped to the same representation if and only if
	$\underline{a}-\underline{b}\in\ker\ell$. 
	Let $L$ be the kernel of $\ell$. 
	We claim that $[\ZZ^\nu : L] = \card{G_1}$.
	Let $\mu_n(\kk) \subseteq \kk^*$ be the group of $n$-th roots of unity.	
	The order of every element of $G_1$ is divisible by $n$, so $G_1$ is in fact a subgroup of $\mu_n(\kk)^\nu$.
	Further, every $\rho_{\underline{a}}$ takes values in $\mu_n(\kk)$. 
	Note that $\mu_n(\kk)$ is (not canonically) isomorphic to $\ZZ/n$ and let $G_2 \subset (\ZZ/n)^\nu$ be the image of $G_1$ under such an isomorphism.
	Then the map $\ell$ corresponds to the map
	\begin{align*}
		\tilde{\ell}: \ZZ^\nu &\rightarrow \Hom_{\ZZ}(G_2, \ZZ/n),\\
		\underline{a} &\mapsto \Big(g\mapsto\sum_i a_i g_i\Big), 
	\end{align*}
	and the claim follows from the preceding Lemma \ref{lemmacyclicrepresentation1}.
	\par Now, since $[\ZZ^\nu : L] = \card{G_1}$ the lattice $L$ divides $\ZZ^{\nu}$ into $\card{G_1}$ cosets.
	On the other hand, there are exactly $\card{G_1}$ irreducible representations, so there exists $a_0\in\ZZ^{\nu}$ such that $\ell(a_0)$ is the representation $V'$ which completes the proof.
\end{proof}

\noindent Now we are ready to prove Theorem \ref{theoremcyclicrepresentation}.

\begin{proof}[Proof of \Cref{theoremcyclicrepresentation}]
	Let $\nu:=\dim_{\kk}V$ and let $L \subseteq \ZZ^\nu$, $a_0 \in \ZZ^\nu$ be as in  Lemma \ref{lemmacyclicrepresentation2}.
	Further, let $\Delta \subset \RR^\nu$ denote the simplex spanned by the origin and the unit vectors. 
	By Lemma \ref{lemmacyclicrepresentation2}, $\sum_{q=0}^N \dim_\kk \Sym^q(V)$ equals the number of lattice points in $N \Delta = \set{N p \with p \in \Delta}$,
	and $\sum_{q=0}^N \mult(V', \Sym^q(V))$ equals $\card{N\Delta \cap (a_0+L)}$.
	
	By choosing a basis for $L$ we find a $\nu\times \nu$ matrix $A$ of full rank, such that $L = A \ZZ^\nu$.
	Note that by the elementary divisor theorem, $A$ can be diagonalized and we see that
	$\det A = [\ZZ^\nu : L]$. Now we compute:
	\begin{align*}
		\lim_{N\rightarrow \infty} \frac{\sum_{q=0}^N\mult(V', \Sym^q(V))}{\sum_{q=0}^N\dim_\kk \Sym^q(V)} &= 
		\lim_{N\rightarrow \infty} \frac{\card{N\Delta \cap (a_0+L)}}{\card{N\Delta \cap \ZZ^\nu}} \\
		&= \frac{ \lim_{N\rightarrow \infty} 1/N^\nu\card{\left(N A^{-1}\Delta \cap (A^{-1}a_0+\ZZ^\nu)\right)}}{ \lim_{N\rightarrow \infty} 1/N^\nu\card{N\Delta \cap \ZZ^\nu}} \\
		& 
		= \frac{\Vol(A^{-1}\Delta)}{\Vol(\Delta)} 
		= \frac{\det(A^{-1})\Vol(\Delta)}{\Vol(\Delta)} = \frac{1}{\det A} \\
		&= \frac{1}{[\ZZ^\nu:L]} = \frac{1}{\card{G}}.
	\end{align*}
\end{proof}

We close with the proof of Theorem \ref{thm:gensymsig}.

\begin{proof}[Proof of \Cref{thm:gensymsig}]
Let $N$ be the MCM $R$-module $\Syz^2_R(\kk)$, and we fix $V:=\Ac'(N)$, and $V':=\Ac'(M)$.
From Corollary \ref{cor:auslander} we have $\rank_R \Sym_R^q(N)^{**} = \dim_\kk \Sym_\kk^q(V)$, and $\mult(M, \Sym_R^q(N)^{**}) = \mult( V',  \Sym_\kk^q(V))$.
By Theorem \ref{Auslandertheorem} $V'$ is an irreducible representation of $G$ and by Lemma \ref{lemmacyclicrepresentationfaithful} $V$ is a faithful representation of $G$.  
Thus, an application of Theorem \ref{theoremcyclicrepresentation} concludes the proof.
\end{proof}

\section*{Acknowledgements}
The authors thank H. Brenner for suggesting and discussing the topic of the present note.

\printbibliography

\end{document}